\documentclass[10pt, english]{amsart}
\usepackage{amssymb,amsthm,amsmath,amstext}
\usepackage{mathrsfs}
\usepackage{bm}
\usepackage[utf8]{inputenc}

\usepackage[all]{xy}

\makeatletter
\ifnum\@ptsize=0  \fi \ifnum\@ptsize=2  \fi 

\title{The double Cayley Grassmannian}
\author{Laurent Manivel}
\date{\today}

\address{Institut de Math\'ematiques de Toulouse ; UMR 5219, Universit\'e de Toulouse \& CNRS, F-31062 Toulouse Cedex 9, France}
\email{manivel@math.cnrs.fr}


\theoremstyle{plain}
\newtheorem{theorem}{Theorem}
\newtheorem*{conjecture*}{Conjecture}
\newtheorem{prop}[theorem]{Proposition}
\newtheorem{lemma}[theorem]{Lemma}
\newtheorem{coro}[theorem]{Corollary}

\def\CC{{\mathbb{C}}}
\def\RR{{\mathbb{R}}}
\def\OO{{\mathbb{O}}}

\def\PP{{\mathbb{P}}}
\def\QQ{{\mathbb{Q}}}\def\ZZ{{\mathbb{Z}}}
\def\SS{{\mathbb{S}}}
\def\cO{{\mathcal{O}}}
\def\cE{{\mathcal{E}}}
\def\cF{{\mathcal{F}}}
\def\cG{{\mathcal{G}}}
\def\cL{{\mathcal{L}}}
\def\cU{{\mathcal{U}}}
\def\cV{{\mathcal{V}}}
\def\ra{{\rightarrow}}
\def\lra{{\longrightarrow}}
\def\ft{{\mathfrak t}}
\def\fg{{\mathfrak g}}

\begin{document}

\begin{abstract}
We study the smooth projective symmetric variety of Picard number one
that compactifies the exceptional complex Lie group $G_2$, by 
describing it in terms of vector bundles on the spinor variety of 
$Spin_{14}$. We call it the double Cayley Grassmannian because quite
remarkably, it 
exhibits very similar properties to those of the Cayley Grassmannian
(the other symmetric variety of type $G_2$), but doubled in the certain 
sense. We deduce among other things that all smooth projective symmetric 
varieties of Picard number one are infinitesimally rigid.
\end{abstract}

\maketitle

\section{Introduction}

Symmetric spaces have been of constant interest since their classification by Elie Cartan 
in 1926. In complex algebraic geometry, projective symmetric varieties of Picard number 
one have been classified by Alessandro Ruzzi in 2011 \cite{ruzzi}. Some of them are in 
fact homogeneous under their full automorphism group. Some others are just hyperplane 
sections of homogeneous spaces. 

The two remaining ones are more mysterious, among other 
things because of their connections with the exceptional group $G_2$. These connections 
prompted us to call the first of them the {\it Cayley Grassmannian}, and denote it $CG$;
its geometry and its cohomology (including its small quantum cohomology) were studied in 
\cite{cayley, bm}. The second one is the subject of the present paper; we will call it the 
{\it double Cayley Grassmannian}, and denote it $DG$. 

This terminology is supported 
by the observation that many important properties of $CG$ are also observed for $DG$, 
but {\it doubled} in a certain way. Let us give an overview of a few of them, first for 
the Cayley Grassmannian:
\begin{enumerate}
\item $CG$ compactifies $G_2/SL_2\times SL_2$, acted on by $G_2$,
\item $CG$ parametrizes four dimensional subalgebras of the complex octonion algebra $\OO$,  
\item $CG$ can be described as the zero locus of a general section of a rank $4$ homogeneous 
vector bundle  on the Grassmannian $G(4,V_7)$, where $V_7\simeq Im\OO$ is the natural 
representation of $G_2$, 
\item its linear span in the Pl\"ucker embedding is $\PP(\CC\oplus S^2V_7)$, 
\item its $G_2$-equivariant Hilbert series is $(1-t)^{-1}(1-tV_{2\omega_1})^{-1}
(1-t^2V_{2\omega_2})^{-1}$, 
\item its topological Euler characteristic is $\chi_{top}(CG)={6\choose 2}$,
\item $CG$ admits three orbits under the action of $G_2$, the complement of the open 
one being a hyperplane section, and the closed one being the quadric $\QQ_5$, 
\item if we blowup the closed orbit, we obtain the wonderful compactification of 
$G_2/SL_2\times SL_2$, with the two exceptional divisors
$$E\simeq \PP(Sym^2C)\ra\;\QQ_5 \qquad \mathrm{and} \qquad F\simeq \PP(Sym^2N)\ra X_{ad}(G_2),$$
where $\QQ_5\simeq G_2/P_1$ and $X_{ad}(G_2)\simeq G_2/P_2$ are the two generalized Grassmannians 
of $G_2$, with their $G_2$-homogeneous rank two vector bundles: the {\it Cayley bundle} $C$ over 
$\QQ^5$ and the {\it null bundle} $N$ over  $X_{ad}(G_2)$.
\end{enumerate}
We find it quite remarkable that the double Cayley Grassmannian $DG$ exhibits the very
same properties, in the following "doubled" form:
\begin{enumerate}
\item $DG$ compactifies $G_2$, acted on by $G_2\times G_2$, 
\item $DG$ parametrizes eight dimensional subalgebras of the complex bioctonion algebra $\OO\otimes \CC$,  
\item $DG$ can be described as the zero locus of a general section of a rank $7$ homogeneous 
vector bundle on the spinor variety $\SS_{14}=Spin_{14}/P_7$,
\item its linear span in the spinorial embedding is $\PP(\CC\oplus V_7\otimes V'_7)$, where 
$V_7$ and $V'_7$ are the natural representations of the two copies of $G_2$,
\item its 
equivariant Hilbert series is $(1-t)^{-1}(1-tV_{\omega_1+\omega'_1})^{-1}
(1-t^2V_{\omega_2+\omega'_2})^{-1}$, 
\item its topological Euler characteristic is $\chi_{top}(DG)=6^2$,
\item $DG$ admits three orbits under the action of $G_2\times G_2$, the complement of the open 
one being a hyperplane section, and the closed one being $\QQ_5\times\QQ_5$, 
\item if we blowup the closed orbit, we obtain the wonderful compactification of 
$G_2$, with the two exceptional divisors
$$E\simeq \PP(C\boxtimes C')\ra\;\QQ_5\times\QQ_5 \quad \mathrm{and} \quad 
F\simeq \PP(N\boxtimes N')\ra X_{ad}(G_2)\times X_{ad}(G_2).$$
\end{enumerate}
The main body of the paper will be devoted to the proof of these properties. In a sense, 
the whole story is hidden in the observation, already found in \cite[Proposition 40]{sk}, 
that $Spin_{14}$ acts almost transitively on the projectivization of its half-spin 
representations, with generic stabilizer $G_2\times G_2$. An important consequence 
is the {\it multiplicative double-point property} used in \cite{am} in order to obtain
a remarkable matrix factorization of the octic invariant of these
representations. We will use this property in an essential way in order to understand
the geometry of $DG$. 

We have not been able to describe its cohomology, partly because the number of classes 
is too big. In principle one should be able to 
deduce it from the cohomology of its blowup along the closed orbit, which should 
be  accessible using \cite{brioninfinity, bj, strickland}. What we have been
able to check is that $DG$ is infinitesimally rigid, a question motivated by a longstanding
interest for the rigidity properties  of homogeneous and quasi-homogeneous spaces (see for example 
\cite{hm, fuman, pk}). This concludes the proof of the following statement:

\medskip\noindent {\bf Proposition}. 
{\it 
Every smooth projective symmetric variety of Picard number one is infinitesimally rigid.} 

\medskip
Along the way, when discussing the geometry of $DG$, we will meet two varieties, admitting
an action of $G_2\times G_2$, which are Fano manifolds of Picard number one, and as
such would deserve special consideration (see Propositions $16$ and $18$). This 
illustrates, once again, the amazing wealth of beautiful geometric objects related 
to the exceptional Lie groups. 

\medskip\noindent {\it Acknowledgements}. We thank Sasha Kuznetsov, Kyeong-Dong Park,
Boris Pasquier and Nicolas Perrin for their useful comments and hints.

\section{Geometric description}

\subsection{Fano symmetric varieties of Picard number one}
Ruzzi proved in \cite{ruzzi} that there exist exactly six smooth projective symmetric varieties 
of Picard number one which are not homogeneous. One of them is a completion of $G_2$, considered
as the symmetric space $(G_2\times G_2)/G_2$. From \cite{ruzzi} we can extract the following 
information. 
\begin{enumerate}
\item The symmetric space $G_2$ admits a unique smooth equivariant completion with Picard 
number one, that we denote $DG$.
\item The connected automorphism group of $DG$ is $G_2\times G_2$; it has index two inside the full 
automorphism group.
\item Under the action of $G_2\times G_2$, the variety $DG$ has exactly three orbits: the open
one, a codimension one orbit $\cO_1$, and a closed orbit $\cO_4\simeq \QQ_5\times\QQ_5$. 
The closure $D$ of $\cO_1$ is singular along $\cO_4$. 
\item The blow up of $DG$ along its closed orbit is the wonderful compactification of $G_2$.
\item Consider the spinor variety $\SS_{14}\subset\PP\Delta$, the closed $Spin_{14}$-orbit 
inside a projectivized half-spin representation; then $DG$ can be realized as a linear 
section of $\SS_{14}$ by a linear subspace of codimension $14$. 
\end{enumerate}
The last statement provides a geometric realization of $DG$ which is not so useful, 
since the linear subspace is highly non transverse (note that $\SS_{14}$ has dimension $21$). 
Our first observation is that a more satisfactory description can be given in terms of 
vector bundles.

\subsection{Octonionic factorization} We will need some extra information on 
half-spin representations. Let $V_{14}$ be a fourteen dimensional complex vector space endowed with a non degenerate quadratic form. Let $\Delta$ be one of the half-spin representations of 
$Spin_{14}$. Its dimension is $64$, and the action of the $91$-dimensional group $Spin_{14}$ 
on $\PP\Delta$ is prehomogeneous.

Recall that 
if we fix a maximal isotropic subspace $E$ of $V_{14}$, we can identify the half-spin 
representation $\Delta$ with the even part $\wedge^+E$ of the exterior algebra $\wedge^\bullet E$. For $e_1,\ldots , e_7$ a besis of $E$, let us denote $e_{ij}=e_i\wedge e_j$, and so on. 
A general element of $\Delta$ is then 
$$z=1+e_{1237}+e_{4567}+e_{123456}$$
The stabilizer of $z$ in $Spin_{14}$ is locally isomorphic with $G_2\times G_2$ (see
\cite[Proposition 40]{sk} or \cite[Proposition 2.1.1]{am}). 
The following statement was proved in \cite{am}.

\begin{prop} A general element $z$ of $\Delta$ determines an orthogonal decomposition 
$V_{14}=V_7\oplus V'_7$. This yields a factorization of $\Delta$ as $\Delta_8\otimes\Delta'_8$, 
for $\Delta_8$ and $\Delta'_8$ the spin representations of $Spin(V_7)$ and $Spin(V'_7)$, such that $z=\delta\otimes\delta'$ for some general $\delta\in\Delta_8$ and $\delta'\in\Delta'_8$. \end{prop}

Explicitely, for $z=1+e_{1237}+e_{4567}+e_{123456}$ we get an orthogonal decomposition 
of $V_{14}$ as the direct sum of the two spaces
$$V_7=\langle e_1,e_2,e_3,f_1,f_2,f_3,e_7-f_7\rangle , \qquad  
V'_7=\langle e_4,e_5,e_6,f_4,f_5,f_6,e_7+f_7\rangle ,$$
such that each copy of $G_2$ acts naturally on one of them, and trivially on the other one. 
Moreover $\delta=1+e_{123}$ and $\delta'=1+e_{456}$. The stabilizer of $\delta$ (resp. $\delta'$)
in $Spin(V_7)$ (resp. $Spin(V'_7)$) is the corresponding $G_2$. 

Let us analyze how $\Delta$ decomposes as a $G_2\times G_2$-module. As a 
$Spin_7\times Spin_7$-module, we have just mentionned that 
$\Delta$ is a tensor product $\Delta_8\otimes \Delta'_8$
of eight-dimensional spin representations. Moreover we can identify $\Delta_8$ with
$\wedge^\bullet A$ and $\Delta'_8$ with $\wedge^\bullet A'$, where $A=\langle e_1, e_2, e_3\rangle$ and $A'=\langle e_4, e_5, e_6\rangle$. Now, the restriction of $\Delta_7$
to $G_2$ decomposes as $\CC\oplus V_7$, so that finally 
$$\Delta \simeq V_7\otimes V'_7\oplus V_7\oplus V'_7\oplus \CC.$$
The result of \cite{ruzzi} is that $DG$ is the (highly non transverse) intersection of 
$\SS_{14}$ with $\PP D_z$, where $D_z=V_7\otimes V'_7\oplus \CC\subset\Delta$. 

The orthogonal to $D_z$ can be described as follows.
The Clifford multiplication yields a morphism $V_{14}\otimes \Delta\rightarrow \Delta^\vee$. 
The image of $V_{14}\otimes z$ is a subspace $L_z$ of $\Delta^\vee$, of dimension $14$, 
which must be stable under $G_2\times G_2$. In particular it must coincide with the 
orthogonal of $D_z$. We can explicitely determine this subspace by computing a basis:
$$\begin{array}{rclcrcl}
e_1.z & = & e_1+e_{14567}, & & f_1.z & = & e_{237}+e_{23456}, \\
e_2.z & = & e_2+e_{24567}, & & f_2.z & = & -e_{137}-e_{13456}, \\
e_3.z & = & e_3+e_{34567}, & & f_3.z & = & e_{127}+e_{12456}, \\
e_4.z & = & e_4-e_{12347}, & & f_4.z & = & e_{567}-e_{12356}, \\
e_5.z & = & e_5-e_{12357}, & & f_5.z & = & -e_{467}+e_{12346}, \\
e_6.z & = & e_6-e_{12367}, & & f_6.z & = & e_{457}-e_{12345}, \\
e_7.z & = & e_7+e_{1234567}, & & f_7.z & = & -e_{123}-e_{456}.
\end{array}$$

\smallskip
\subsection{Spinorial interpretation}
Let us denote by $\cL$ the very ample line bundle that defines the embedding of the spinor 
variety $\SS_{14}\subset\PP\Delta$. Recall that $\Delta$ is one of the half-spin representations of 
$Spin_{14}$, and its dimension is $64$. The spinor variety $\SS_{14}$ parametrizes one of
the two families of maximal isotropic spaces in $V_{14}$, and the square $\cL^2$ defines the Pl\"ucker embedding 
$$\SS_{14}\hookrightarrow G(7,V_{14})\subset\PP(\wedge^7V_{14}).$$ The tautological bundle on 
$G(7,V_{14})$ restricts to a rank seven vector bundle $\cU$ on $\SS_{14}$, such 
that $\det (\cU)=\cL^{-2}$. Moreover, $\cU\otimes \cL$ is an irreducible homogeneous vector 
bundle, and by the Borel-Weil theorem, 
$$H^0(\SS_{14}, \cL)=\Delta^\vee \qquad \mathrm{and}\qquad 
H^0(\SS_{14}, \cU\otimes \cL)=\Delta.$$
Since $\cU\otimes \cL$ is irreducible and admits non zero sections, it is automatically 
globally generated. So a general section vanishes along a codimension seven subvariety 
of $\SS_{14}$. Note that this zero locus is (locally) constant up to projective isomorphism,
since $Spin_{14}$ acts on $\PP \Delta$ with an open orbit (whose complement is a degree $7$ hypersurface, see \cite{am} for more details).

\begin{prop}
The zero locus of a general section of the vector bundle $\cU\otimes \cL$ on $\SS_{14}$
is projectively isomorphic with $DG$. 
\end{prop}

\proof Let $z$ be a general element of $\Delta$, and $s_z$ the associated 
section of $\cU\otimes \cL$. Let $y$ be a pure spinor; in other words,  $[y]$ is a point
of $\SS_{14}$. Then $s_z([y])$ is a linear homomorphism from $\cL^\vee_{[y]}=\CC y$ to
$\cU_{[y]}$. The latter is the subspace of $V_{14}$ characterized as 
$$\cU_{[y]} = \{v\in V_{14}, \; v.y=0\},$$
where $v.y\in\Delta^\vee$ denotes the Clifford product of the vector $v$ by the spinor $y$
(recall that the fact that $\cU_{[y]}$ is maximal isotropic is equivalent to $y$ being a 
pure spinor \cite{chevalley}). 
We claim that $s_z([y])$ is defined by the following formula:
$$s_z([y])(u)=\langle z, u.y\rangle, \qquad u\in V_{14}.$$
Note that the right hand side is a linear form in $u\in V_{14}$ that certainly 
vanishes on $\cU_{[y]}$. Since it is maximal isotropic, $\cU_{[y]}\simeq \cU_{[y]}^\perp$.
So the right hand side really defines an element of $\cU_{[y]}$, depending linearly 
on $y\in [y]$, as required. 

We have therefore defined a non trivial equivariant morphism from $\Delta$ to 
$H^0(\SS_{14}, \cU\otimes \cL)$. 
By the Schur Lemma, it must be an isomorphism, and the same one up to scalar as the one 
provided by the Borel-Weil theorem. 

So the zero-locus of $s_z$ is the set of points $[y]\in \SS_{14}$ such that 
$$\langle z, u.y\rangle=\langle u.z, y\rangle=0 \quad \forall u\in V_{14}.$$
In other words, set theoretically it is the intersection of $\SS_{14}$ with
the orthogonal to the fourteen dimension subspace $V_{14}.z\subset\Delta^\vee$. 
This is exactly Ruzzi's description, and we are done. \qed

\begin{coro}
$DG$ is a prime Fano manifold of dimension $14$ and index $7$.
\end{coro}

\proof 
$\SS_{14}$ has index $12$, while $\det(\cU\otimes \cL)=
\det(\cU)\otimes \cL^7=\cL^5$. Of course the restriction of $\cL$ cannot be divisible 
since by Kobayashi-Ochiai it cannot be bigger that $15$, and $DG$ would be a quadric if it 
was equal to $14$.\qed 

\medskip Recall that the Chow ring of $\SS_{14}$ has an integral basis of Schubert 
classes $\tau_\mu$ indexed by strict partitions $\mu = (\mu_1>\cdots >\mu_m>0)$, 
with $\mu_1\le 6$. In particular $\tau_1$ is the hyperplane class, and the Pieri 
formula states that 
$$\tau_\mu\tau_1 = \sum_{\nu}\tau_\nu,$$
where the sum is over all strict partitions $\nu$ obtained by adding one to some 
part of $\mu$ (or adding a part equal to one). There is a more general version 
for the product of a Schubert class by a special class $\tau_k$, with multiplicities 
given by certain powers of two \cite{hillerboe}. A consequence is that the Chow 
ring of  $\SS_{14}$ is generated, over the rationals, by the three special classes 
$\tau_1, \tau_3, \tau_5$.

\begin{coro}
The fundamental class of $DG$ in the Chow ring of $\SS_{14}$ is 
$$[DG]=c_7(\cU\otimes\cL)= \tau_{61}+\tau_{52}+\tau_{43}+\tau_{421}=
2\tau_1\tau_3^2+2\tau_1^2\tau_5-6\tau_1^4\tau_3+3\tau_1^7.$$
\end{coro}

\proof  By the Thom-Porteous formula $[DG]=c_7(\cU\otimes\cL)$. Since 
$$c_7(\cU\otimes\cL)=\sum_{i=0}^7c_i(\cU)c_1(\cL)^{7-i},$$
a repeated application of the Pieri formula yields the result. \qed 

\medskip Another direct application is to rigidity questions, which attracted strong 
interests for homogeneous spaces and their subvarieties \cite{hm, fuman}. 

\begin{prop}
$DG$ is infinitesimally rigid. 
\end{prop}

\proof Since $DG$ is Fano, its deformations are non obstructed and we just need
to prove that $H^1(T_{DG})=0$. Then the usual computations with the Koszul complex
and the Borel-Weil-Bott theorem yield the result. Indeed, the Koszul complex 
takes the form 
$$0\lra \cL^{-5}\lra \cU\otimes\cL^{-4}\lra \cdots \lra  \cU^\vee\otimes\cL^{-1}
\lra \cO_{\SS_{14}}\lra \cO_{DG}\lra 0.$$

\noindent {\it First step}. We first prove that $H^1(T\SS_{14|DG})=0$ by tensoring the Koszul complex 
with $T\SS_{14}=\wedge^2\cU^\vee$,  and then by checking that for any integer $k$, with $0\le k\le 7$, 
the cohomology group
$$H^{k+1}(\SS_{14},\wedge^2\cU^\vee\otimes\wedge^k\cU^\vee\otimes\cL^{-k})=0.$$
For $k=0$ we just get the irreducible bundle $\wedge^2\cU^\vee$, which is globally
generated and has no higher cohomology by the Bott-Borel-Weil theorem. For $k>0$, 
the tensor product $ \wedge^2\cU^\vee\otimes\wedge^k\cU^\vee$ is the direct sum of 
at most three irreducible homogeneous bundles, of respective weights 
$\lambda_k=\epsilon_1+\cdots +\epsilon_{k+2}$ (for $k\le 5$), $\mu_k=2\epsilon_1+\epsilon_2+\cdots +
\epsilon_{k+1}$ (for $1\le k\le 6$) and  $\nu_k=2\epsilon_1+2\epsilon_2+\epsilon_3+\cdots +
\epsilon_{k}$ (for $k\ge 2$). Here we made the usual choice of positive roots $\epsilon_i\pm
\epsilon_j$ for $1\le i<j\le 7$, where $(\epsilon_1, \ldots , \epsilon_7)$ is an orthonormal basis. Following the Bott-Borel-Weil theorem, these bundles twisted by 
$\cL^{-k}$ are acyclic if we can find roots $\varphi_k, \chi_k, \psi_k$ such that 
$$\langle \lambda_k-k\omega_7+\rho , \varphi_k\rangle = \langle \mu_k-k\omega_7+\rho , \chi_k\rangle = 
\langle \nu_k-k\omega_7+\rho , \psi_k\rangle =0,$$
where $\rho$ denotes the sum of the fundamental weights, and $\omega_7=\frac{1}{2}(\epsilon _1+\cdots +\epsilon_7)$. 
We will look for a root of the form $\varphi_k=\epsilon_i+\epsilon_j$, with $1\le i<j\le 7$, so that we always have $\langle \omega_6 , \varphi_k\rangle = 1$.  
Then the vanishing condition becomes $i+j+k=14+\delta$, with $\delta=2$ for $j\le k+2$, $\delta=1$ for $i\le k+2<j$, 
and $\delta=0$ for $k+2<i$. 
Solutions do exist for any $k=1,\ldots, 5$: take respectively $(i,j)=(6,7), (5,7), (5,7), (4,7), (5,6)$.  
Similarly we can choose the root $ \chi_k$, for $k=1,\ldots ,6$, to be again of the form $\epsilon_i+\epsilon_j$ with $(i,j)=(6,7), (5,7), (5,6), (4,7), (3,7), (3,7)$. 
Finally for the root $ \psi_k$ we can choose $\epsilon_i+\epsilon_j$ with $(i,j)=(5,7), (5,6), (4,7), (3,7), (4,6), (3,6)$ for $k=2,\ldots ,7$.

\smallskip
\noindent {\it Second step}. Then we need to compute $H^0(\cU\otimes\cL_{|DG})$. Using the 
same techniques as in the previous step, we check that the restriction morphism 
$$H^0(\cU\otimes\cL)\lra H^0(\cU\otimes\cL_{|DG})$$ is surjective, with kernel generated
by the section that defines $DG$. In other words, $H^0(\cU\otimes\cL_{|DG})\simeq \Delta/\CC z$.

\smallskip
\noindent {\it Third step}. We conclude the proof by checking that the morphism 
$$H^0(T\SS_{14|DG})\lra H^0(\cU\otimes\cL_{|DG})$$ is surjective.
For this we simply observe that it factorizes the morphism from $H^0(T\SS_{14})\simeq 
\mathfrak{spin}_{14}$
to $\Delta/\CC z$ given by $X\mapsto Xz$ mod $\CC z$. Finally, the surjectivity of the latter
morphism is equivalent to the fact the orbit of $[z]$ is open in $\PP (\Delta)$. 
\qed 

\medskip As we already mentionned in the introduction, 
this implies that all the smooth projective symmetric
varieties of Picard number one are infinitesimally rigid (see \cite{pk}). 

\medskip\noindent {\it Question}. Is $DG$ globally rigid? There are very nice examples of 
linear sections  (of codimension two and three) of the ten dimensional spinor variety 
$\SS_{10}$, which are defined by the generic point of a representation with an open 
orbit, and turn our for this reason to be locally rigid. However, they are not globally rigid
because the generic points of some smaller orbits still define smooth sections, but of 
a different type \cite{kuzspin10,fuman}. In our case,  what does happen if we replace 
the general point $z$  of $\Delta$ by a general point of its invariant octic divisor? 
Since this divisor is the dual to the spinor variety in the dual representation, the zero-locus
of a section defined by such a point should contain a special $\PP^6$; is it its
singular locus? An explicit representative is 
$$z_1 =  1+e_{1237}+e_{1587}+e_{2467}+e_{123456}.$$
In the case of the Cayley Grassmannian $CG$, general sections from the exceptional divisor 
define a $\PP^3$ which is singular inside the zero-locus, so there is no immediate
obstruction to global rigidity. Up to our knowledge the question of the global rigidity 
of $CG$ remains open.

\section{Octonionic interpretations}
Consider the real algebra $\CC\otimes_{\RR}\OO_\RR$, with the obvious product. This is 
called the algebra of complex octonions, or bioctonions. Of course
it is no longer a division algebra, but it is still what is called a  {\it structurable algebra} \cite{allison}. 
We will consider this algebra with complex coefficients: in other words, we complexify once more. 

\begin{prop}
The double Cayley Grassmannian $DG$ parametrizes the eight-dimensional 
isotropic subalgebras of the complexified bioctonions.
\end{prop}

The main point is that complexifying the complex numbers, we just get the algebra
$\CC\oplus\CC$. Indeed, if we denote by $i$ and $I$ the roots of $-1$ in our two
copies of $\CC$, then $E=(1+iI)/2$ and $F=(1-iI)/2$ are such that $E+F=1$, $EF=FE=0$,  
and $E^2=E$ and $F^2=F$. Hence an isomorphism
$$\CC\otimes_{\RR}\CC\otimes_{\RR}\OO_\RR\simeq \OO\oplus\OO.$$
An eight dimensional subspace of 
$\OO\oplus \OO$, which is transverse to this decomposition, can be written 
as the graph $\Gamma_g$ of some $g\in GL(\OO)$. Moreover, it contains the unit element 
if and only if $g(1)=1$. And it is a subalgebra if and only if $g$ belongs to $G_2$. 
It is then generated by the unit element, and its intersection $L_g$ with $V_{14}=Im\OO\oplus 
Im\OO$. Note that $\Gamma_g$ (respectively $L_g$) is isotropic with respect
to the difference of the octonionic norms on the two copies of $\OO$ (respectively 
$Im\OO$). This yields an 
embedding of $G_2$ inside $Spin_{14}$, whose closure is exactly $DG$. 

So $DG$ parametrizes a certain family of subspaces of the bioctonions. 
These spaces must be isotropic subalgebras, since this condition is closed. 
So let us consider such a subalgebra $A$, and suppose it defines a point 
of $DG$, not on the open orbit. Let $K,K'$ denote the kernels of the projections
to the two copies of $\OO$. They must be positive dimensional subspaces 
of $Im\OO$, totally isotropic, and such that $KK\subset K$ and $K'K'\subset K'$. 
In particular $\CC 1+K$ and $\CC 1+K'$ are subalgebras of $\OO$. Let $k=\dim K$
and $k'=\dim K'$. These are invariants of the $Spin_{14}$-action, and since this group
has only three orbits on $DG$, there are at most two possibilities for the pair
$(k,k')$, apart from the generic case $(k,k')=(0,0)$. 

 \medskip\noindent {\bf First case:} $(k,k')=(3,3)$. Then $\CC 1+K$ and $\CC 1+K'$ 
are four dimensional subalgebras of $\OO$. By \cite[Proposition 2.7]{cayley}, 
the isotropic four dimensional subalgebras of $\OO$ are parametrized by the 
quadric $\QQ_5=G_2/P_1$. Explicitly, if $\ell$ is an isotropic line in $Im\OO$, 
then $K_\ell=\ell\OO\cap Im\OO$ is such a subalgebra, and they are all of this 
type. 

When $K$ and $K'$ are given, then
$K\oplus K'$ is isotropic of dimension six, so it is contained in exactly two 
maximal isotropic subspaces of $V_{14}$, one in each family. In particular there
is exactly one in $\SS_{14}$. This defines an 
embedding of $\QQ_5\times\QQ_5$ inside $\SS_{14}$. Since this is the unique 
$G_2\times G_2$-equivariant embedding of $\QQ_5\times\QQ_5$ in $\PP\Delta$, it must factor 
through $DG$.

\medskip\noindent {\bf Second case:} $(k,k')=(2,2)$.
We will show how to construct examples of this type. Since we know there is only one 
orbit which is neither closed nor open, this will necessarily provide us with representatives
of this intermediate orbit $\cO_1$. We start with two null-planes $N$ and $N'$. Recall 
that $\CC 1\oplus N^\perp$ is a six dimensional subalgebra of $\OO$, a copy of 
the sextonion subalgebra \cite{sext}. Moreover it contains $H$, a copy of the quaternion algebra
transverse to $N$. (Over the complex numbers, the quaternion algebra is just an
algebra of rank two matrices, and $N$ is isomorphic with its two-dimensional 
simple module.) Let us also choose $H'$ in $\CC 1\oplus N'^\perp$, transverse to 
$N'$. Consider 
$$A =(N,0)\oplus (0,N')\oplus \Delta_h,$$
where $\Delta_h$ is the graph of some morphism $\delta$ from $H$ to $H'$. 
Then $A$ is an isotropic subalgebra of the bioctonions if and only if $\delta$
is an algebra isomorphism. 

We claim that $A$ belongs to $DG$. Because of the $G_2\times G_2$-equivariance,
it is enough to exhibit just one such $A$ that does belong to $DG$. To do this we 
shall start from an explicit null plane in $Im\OO$. Let $u_1, \ldots , u_7$ be an
orthonormal basis of $Im\OO$, whose multiplication rule is encoded in a Fano plane,
as in \cite{cayley}. Then for example, $N=\langle u_1+iu_2, u_4-iu_5\rangle $ is a null-plane. 
It is convenient to reindex this basis by letting $u_1=v_{-1}, u_2=v_{2}, u_3=v_{-3}, 
u_4=v_{1}, u_5=v_{-2}, u_6=v_{3}, u_7=v_{0}$. Then we may suppose that the 
transformation rule between the basis $v_{-3},v_{-2},v_{-1}, v_0, v_1, v_2, v_3$ and 
$e_1, e_2,e_3, f_1, f_2, f_3, e_7-f_7$ is given by 
$$v_k=\frac{1}{\sqrt{2}}(e_k+f_k), \quad v_{-k}=\frac{i}{\sqrt{2}}(e_k-f_k), 
\quad v_0=\frac{i}{\sqrt{2}}(e_7-f_7).$$  
After this change of basis, our null-plane of $V_7$ becomes $N=\langle e_1+e_2,
f_1-f_2\rangle$. Similarly, $N'=\langle e_4+e_5, f_4-f_5\rangle$ is a null-plane 
in $V'_7$. 

\smallskip\noindent {\it Remark}. Note the connection with the {\it null triples} of \cite{bh}. 

\begin{lemma}
The three dimensional projective space $\PP(N\otimes N')$ is contained in $DG$.
Moreover a spinor $x\in N\otimes N'$ is of type $(3,3)$ if its tensor rank is one, 
and type $(2,2)$ if its tensor rank is two.
\end{lemma}

\proof We have the following correspondance between vectors in $N\otimes N'$ 
and in $\Delta$:
$$\begin{array}{rcl}
(e_1+e_2)\otimes (e_4+e_5) & \mapsto & y_1=(e_1+e_2)(e_4+e_5), \\
(e_1+e_2)\otimes (f_4-f_5) & \mapsto & y_2=(e_1+e_2)(e_4+e_5)e_6e_7, \\
(f_1-f_2)\otimes (e_4+e_5) & \mapsto & y_3=(e_1+e_2)(e_4+e_5)e_3e_7, \\
(e_1+e_2)\otimes (e_4+e_5) & \mapsto & y_4=(e_1+e_2)(e_4+e_5)e_3e_6.
\end{array}$$
This allows to check that $N\otimes N'$ is orthogonal 
to $L_z$. So its projectivization will be contained in $DG$ as soon as it 
only consists in pure spinors. Consider $y=t_1y_1+t_2y_2+t_3y_3+t_4y_4$. 
A straightforward computation shows that $y$ is annihilated by 
$$P_y=\langle e_1+e_2, f_1-f_2, e_4+e_5, f_4-f_5, p_3, p_6, p_7 \rangle , $$
where $p_3=t_4e_6+t_3e_7-t_1f_3, p_6=t_4e_3-t_2e_7+t_1f_6, p_7=t_3e_3+t_2e_6+t_1f_7$.
In particular $y$ is the pure spinor associated (up to scalar) to the maximal isotropic space $P_y$. 
Note moreover that the intersection of $P_y$ with $\langle f_1, \ldots ,f_7\rangle$ has dimension 
equal to two plus the corank of a size three skew-symmetric matrix; in particular this dimension is 
always odd, which means that $y$ is a positive pure spinor. In other words, it is a point of $DG$. \qed 

\medskip
Recall that we denoted by $D$ the closure of the codimension one orbit in $DG$. 
Necessarily, $D$ must be the intersection of $DG$ with the hyperplane $\PP(V_7\otimes V'_7)$. 
Moreover, by the previous lemma $D$ contains the union of the projective spaces 
$\PP(N\otimes N')$, for $N$ and $N'$ null-planes in $V_7$ and $V'_7$. Since 
this union is obviously $G_2\times G_2$-invariant, it has to coincide with $D$. 
(This describes $D$ as the image of a projectivized Kempf collapsing). Moreover,
for the very same reason the closed orbit $\cO_4$ must be the union of the 
rank one elements $\PP N\times \PP N'\subset \PP(N\otimes N')$. Since the 
intersection of two different tensor products $N_1\otimes N'_1$ and $N_2\otimes N'_2$
can only contain elements of rank one (or zero), we deduce the following statement.

\begin{prop}
Suppose that $x$ belongs to $\cO_1$. Then there exists a unique null-plane $N_x$
in $V_7$, and a unique null-plane $N'_x$ in $V'_7$, such that $x$ is contained 
in $\PP(N_x\otimes N'_x)$. Moreover, $x$ has full rank in $\PP(N_x\otimes N'_x)$.
\end{prop}

Geometrically, this means that $\cO_1$ fibers over a product of adjoint varieties
$X_{ad}(G_2)\times X_{ad}(G_2)$, with fiber the complement of a smooth quadric in $\PP^3$.

\section{Postulation} 
Recall that the vertices of the Dynkin diagram $D_7$ are in bijective 
correspondence with the fundamental weights $\omega_i$, or the fundamental 
representations $V_{\omega_i}$ of $Spin_{14}$, for $1\le i\le 7$. 
We use the following indexation:

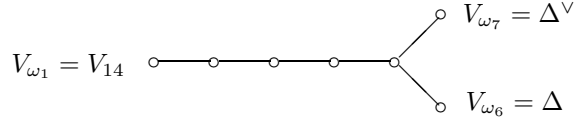
\begin{center}
\setlength{\unitlength}{4mm}

\begin{picture}(20,4)(0,-1.5)
\multiput(5,0)(2,0){5}{$\circ$}
\multiput(5.4,.25)(2,0){4}{$\line(1,0){1.7}$}
\put(13.4,.4){\line(1,1){1.3}} 
\put(13.4,.1){\line(1,-1){1.3}} 
\put(14.55,-1.5){$\circ$}
\put(14.55,1.6){$\circ$}
\put(.5,-.1){$V_{\omega_1}=V_{14}$}
\put(15.6,-1.35){$V_{\omega_6}=\Delta$}\put(15.5,1.55){$V_{\omega_7}=\Delta^\vee$}
\end{picture} 
\end{center}
One way to compute the cohomology groups on $DG$ of $\cL$ and its powers, 
is again to use the Koszul complex 
\begin{equation}
0\lra \wedge^7\cE^\vee \lra \cdots \lra \cE^\vee \lra\cO_{\SS_{14}}\lra \cO_{DG}\lra 0,
\end{equation}
where $\cE=\cU\otimes\cL$. For any $k\ge 0$ and $i\ge 0$, the bundle $\wedge^i\cE^\vee\otimes\cL^k=
\wedge^i\cU^{\vee}\otimes\cL^{k-i}$ is irreducible, with highest weight $\theta_i$ given by 
$\theta_i=(k-i)\omega_7+\omega_i$ for $0\le i\le 5$ (and $\omega_0=0$ by convention), while $\theta_6=(k-5)\omega_7+\omega_6$
and $\theta_7=(k-5)\omega_7$. One easily checks that these weights are either dominant or
singular. By the Bott-Borel-Weil theorem this implies that $\cL^k$ has no higher cohomology. Moreover we can compute the 
dimension of its space of global sections as the alternate sum of  modules whose dimensions
are given by the Weyl dimension formula, as follows: 

$${\scriptstyle \dim V_{k\omega_7}=\frac{(k+1)(k+2)(k+3)^2(k+4)^2(k+5)^3(k+6)^3(k+7)^3(k+8)^2(k+9)^2(k+10)(k+11)}{1\times 2\times 3^2 \times 4^2 \times 5^3 \times 6^3 \times 7^3 \times 8^2 \times 9^2 \times 10\times 11}, }$$ 
$${\scriptstyle \dim V_{(k-1)\omega_7+\omega_1}=\frac{k(k+1)(k+2)^2(k+3)^2(k+4)^3(k+5)^2(k+6)^3(k+7)^2(k+8)^2(k+9)(k+10)(k+11)}{3^2 \times 4^2 \times 5^3 \times 6^2 \times 7^2 \times 8^2 \times 9^2 \times 10\times 11\times 12}, }$$ 
$${\scriptstyle \dim V_{(k-2)\omega_7+\omega_2}=\frac{(k-1)k(k+1)^2(k+2)^2(k+3)^2(k+4)^2(k+5)^3(k+6)^2(k+7)^2(k+8)^2(k+9)(k+11)}{2\times 3^2 \times 4^2 \times 5^2 \times 6^2 \times 7^2 \times 8^2 \times 9^2 \times 10^2\times 11}, }$$ 
$${\scriptstyle \dim V_{(k-3)\omega_7+\omega_3}=\frac{(k-2)(k-1)k^2(k+1)(k+2)^2(k+3)^2(k+4)^3(k+5)^3(k+6)^2(k+7)(k+8)(k+9)(k+10)}{2\times 3^2 \times 5^2 \times 6^2 \times 7^2 \times 8^3 \times 9^2 \times 10\times 11\times 12}, }$$ 
$${\scriptstyle \dim V_{(k-4)\omega_7+\omega_4}=\frac{(k-3)(k-2)(k-1)k(k+1)^2(k+2)^3(k+3)^3(k+4)^2(k+5)^2(k+6)(k+7)^2(k+8)(k+9)}{2\times 3\times 4 \times 5^2 \times 6^3 \times 7^2 \times 8^2 \times 9^2 \times 10\times 11\times 12}, }$$ 
$${\scriptstyle \dim V_{(k-5)\omega_7+\omega_5}=\frac{(k-4)(k-2)(k-1)^2k^2(k+1)^2(k+2)^3(k+3)^2(k+4)^2(k+5)^2(k+6)^2(k+7)(k+8)}{2\times 3\times 4 \times 5^2 \times 6^2 \times 7^2 \times 8^2 \times 9^2 \times 10^2\times 11\times 12}, }$$ 
$${\scriptstyle \dim V_{(k-5)\omega_7+\omega_6}=\frac{(k-4)(k-3)(k-2)(k-1)^2k^2(k+1)^3(k+2)^2(k+3)^3(k+4)^2(k+5)^2(k+6)(k+7)}{2\times 3\times 4^2 \times 5^2 \times 6^3 \times 7^2 \times 8 \times 9^2 \times 10^2\times 11\times 12}, }$$ 
$${\scriptstyle \dim V_{(k-5)\omega_7}=\frac{(k-4)(k-3)(k-2)^2(k-1)^2k^3(k+1)^3(k+2)^3(k+3)^2(k+4)^2(k+5)(k+6)}{1\times 2\times 3^2 \times 4^2 \times 5^3 \times 6^3 \times 7^3 \times 8^2 \times 9^2 \times 10\times 11}.}$$

\begin{prop}
For any $k\ge 0$ and $i>0$, $H^i(DG,\cL^k)=0$. Moreover, 
$$h^0(DG,\cL^k)=\frac{(k+1)(k+2)(k+3)^2(k+4)^2(k+5)(k+6)}{2^{10}3^55^27^2 11} P(k),$$
$P(k)=186k^6+3906k^5
+34441k^4+163184k^3+438545k^2+634858k+388080.$
\end{prop}

\begin{coro}
The degree of $DG\subset\PP^{49}$ is $4836=2^2\times 3\times 13\times 31$.
\end{coro}

This could also have been deduced from the fundamental class of $DG$, by applying repeatedly
the product formula by the hyperplane class. 

\smallskip
Since $DG$ is spherical, it is multiplicity free. As in \cite{cayley}, we can obtain the 
$G_2\times G_2$-module structure of $H^0(DG,\cL^k)$ by restricting to the hyperplane divisor $D$.
Using the projecting bundle structure of its resolution, we get 
$$H^0(D,\cL_D^k)=H^0(X_{ad}(G_2)\times X'_{ad}(G_2), Sym^k(N\otimes N')^\vee).$$
By the Cauchy formula, 
$$Sym^k(N\otimes N') = \bigoplus_{i+2j=k}Sym^iN\otimes (\det N)^j\otimes Sym^iN'\otimes (\det N')^j.$$
Since $Sym^iN^\vee$ is irreducible of highest weight $i\omega_1$, and $\det N^\vee$ of weight $\omega_2$, 
the Borel-Weil theorem yields 
$$H^0(D,\cL_D^k)=\bigoplus_{i+2j=k}V_{i\omega_1+j\omega_2}\otimes V'_{i\omega_1+j\omega_2}.$$
We finally get (to be compared with Proposition 3.6 of \cite{cayley}):

\begin{prop}
The equivariant Hilbert series of the double Cayley Grassmannian is
$$H^{G_2\times G_2}_{DG}(t)=(1-t
)^{-1}(1-tV_{\omega_1+\omega'_1})^{-1}
(1-t^2V_{\omega_2+\omega'_2})^{-1}.$$
\end{prop}

\smallskip
Here we use formally the Cartan multiplication of representations, according to the rule
$V_\mu V_\nu = V_{\mu+\nu}$. Moreover we use it for $G_2\times G_2$, so that $V_{\mu+\nu'}$ is the 
tensor product of the representation $V_{\mu}$ of the first copy of $G_2$, by the representation 
$V_{\nu}$ of the second copy.  
\smallskip

\section{The wonderful compactification of $G_2$}
Recall that the Cayley Grassmannian $CG\subset G(4,V_7)$ has a very similar $G_2$-orbits structure:
a closed orbit $\cO_3\simeq \QQ_5$, a codimension one orbit $\cO_1$ whose closure is a hyperplane 
section $H$ of $CG$, and an open orbit $\cO_0\simeq G_2/SL_2\times SL_2$. Moreover, if we blow-up 
$\cO_3\subset CG$, we get the wonderful compactification of the symmetric space $\cO_0$. 
Since we are in rank two, the proper orbit closures of this wonderful compactification $\overline{CG}$ are 
the two divisors $F$ (the proper transform of $H$), $E$ (the exceptional divisor), and 
their transverse intersection $E\cap F$. The two divisors support smooth projective fibrations:
$$E\simeq \PP(Sym^2C)\ra\QQ_5, \qquad F\simeq \PP(Sym^2N)\ra X_{ad}(\fg_2),$$
where $C$ denotes the so-called {\it Cayley bundle} over $\QQ_5$, and $N$ is the null-plane 
bundle over the adjoint variety  $X_{ad}(G_2)$. Both are rank two irreducible homogeneous bundles.
The latter is the restriction of the tautological bundle for the embedding of $X_{ad}(G_2)$ into $G(2,V_7)$. The former is defined by the conditions that $H^0(C)=0$ and $H^0(C(1))=\fg_2$; its 
first Chern class is the hyperplane class \cite{ott}.

Observe that in particular, $E$ and $F$ both contain a conic fibration, preserved by $G_2$, which must therefore
coincide with the closed orbit $E\cap F$. In fact, this closed orbit is nothing else than
the full flag variety of $G_2$. 

We have the following diagram: 
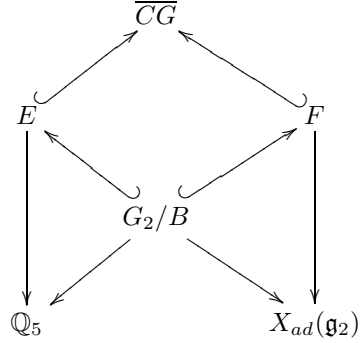
\begin{equation*}\label{blowupCG}
\xymatrix{& \overline{CG} & \\
E \ar[dd]\ar@{^{(}->}[ur] & & F \ar[dd]\ar@{_{(}->}[ul] \\
  & G_2/B \ar@{^{(}->}[ur]\ar@{_{(}->}[ul] \ar[dr] \ar[dl] & \\
  \QQ_5 & & X_{ad}(\fg_2)
}
\end{equation*}

\smallskip
The picture is strickingly similar for the double Cayley Grassmannian. 
Blowing-up the closed orbit $\cO_4\simeq \QQ_5\times \QQ_5$, we get an exceptional
divisor $E$, which is the projectivization of the normal bundle. 

\begin{lemma}
The normal bundle to the closed orbit in $DG$ is $C\otimes C'$. 
\end{lemma}

Moreover the strict transform $F$ of $D$ is the total space of the 
projectivisation of $N\otimes N'$ over $X_{ad}(\fg_2)\times X_{ad}(\fg_2)$. Again 
each of these divisors contains a quadric surface bundle, which must 
coincide with the closed orbit $E\cap F$. In fact this closed 
orbit is nothing else than the product of two copies of the flag variety
of $G_2$. We get the following diagram:
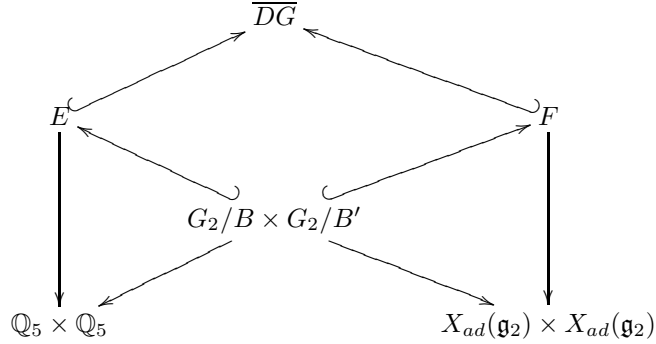
\begin{equation*}\label{blowupDG}
\xymatrix{& \overline{DG} & \\
E \ar[dd]\ar@{^{(}->}[ur] & & F \ar[dd]\ar@{_{(}->}[ul] \\
  & G_2/B\times G_2/B' \ar@{^{(}->}[ur]\ar@{_{(}->}[ul] \ar[dr] \ar[dl] & \\
  \QQ_5\times \QQ_5 & & X_{ad}(\fg_2)\times X_{ad}(\fg_2)
}
\end{equation*}

\medskip 

\proof For a quick check of the Lemma we can argue as follows. The normal bundle $N$ on 
$\QQ_5\times \QQ_5$ we are looking for has rank four, and is by construction 
homogeneous under $G_2\times G_2$, and symmetric with respect to the two
quadrics. In particular it must be constructed from homogeneous bundles 
of rank at most two on the two quadrics. Since there are no non trivial 
extensions between line bundles on $\QQ_5$, this quadric admits only two, up to twists,
$G_2$-homogeneous bundles of rank at most two: the trivial line bundle and the 
Cayley bundle. 

A possibility would be that $N=C(a,b)\oplus C'(b,a)$, where we denote by $C$ and 
$C'$ the two Cayley bundles induced from the two quadrics. But then we would get 
$\det(N)=(2a+2b-1,2a+2b-1)$, while a computation with tangent bundles yields
$\det(N)=(2,2)$. So $N$ must be a twist of  $C\otimes C'$, and since this has
the correct determinant, the twist must be trivial.\qed

\medskip\noindent {\it Remark}. Exactly as in the case of $CG$, there also 
exists another contraction of $\overline{DG}$ to another variety $\hat{DG}$, contracting the
divisor $D$. But the result of this contraction is singular.

\section{Betti numbers} 

In this section we compute the Betti numbers of $DG$. We would like to be able to compute its
cohomology ring. 

\subsection{Torus action}

Let $T$ be a maximal torus of $G_2\times G_2$.

\begin{prop}
The torus $T$ acts on $DG$ with exactly $36$ fixed points, 
all contained in the closed orbit $\QQ_5\times \QQ_5$. 
\end{prop}

\proof Recall that the linear span of $DG$ is the projectivization of $V_7\otimes V'_7$. 
Moreover, $G_2$ acts on $V_7$ with weights $0, \pm\alpha_1, \pm\alpha_2, \pm\alpha_3$ with 
$\alpha_1+\alpha_2+\alpha_3=0$. The weights of the action of $G_2\times G_2$ on $V_7\otimes V'_7
\oplus\CC$
are thus the $\pm\alpha_i, \pm\alpha'_j, \pm\alpha_i\pm\alpha'_j$, all with multiplicity one, 
and $0$ with multiplicity two. Let $W_0$ be the two-dimensional zero weight space. 
To ensure that $T$ acts on $DG$ with finitely many fixed points, the only thing we need to 
check is that the projective line $\PP W_0$ is not contained in $DG$. But this is clear, 
since this line contains $[z]$, which is not contained in $\SS_{14}$ and a fortiori not
in $DG$. 

We claim, more precisely, that:
\begin{enumerate}
\item every $T$-fixed point with non zero weight is contained in $DG$, 
\item $DG\cap \PP W_0$ is empty. 
\end{enumerate}
The first statement is clear, since $\Delta$ being minuscule, each fixed point in $\PP\Delta$
of a maximal torus of $Spin_{14}$ is contained in $\SS_{14}$. Since $T$ is a subtorus of 
a maximal torus $T_+$ of $Spin_{14}$, this remains true for all the $T$-fixed points with
non zero weight, just because they are also $T_+$-fixed points.  

To check the second statement, we may suppose that 
$e_1, e_2, e_3, f_1, f_2$, $f_3, e_7-f_7$ are $T$-eigenvectors in $V_7$, with weights $\alpha_1, 
\alpha_2, \alpha_3, -\alpha_1, -\alpha_2, -\alpha_3, 0$; and similarly for $V'_7$. 
Then the $T$-invariants in $L_z$ are $e_7.z$ and $f_7.z$. From that we deduce that 
$$W_0=\langle 1+e_{123456}, e_{1237}+e_{4567}\rangle .$$
We need to check that $W_0$ contains no pure spinor. Observe that if an element of 
$\Delta$ of the form $1+\omega_2+\omega_4+\omega_6$ is a pure spinor, then $\omega_4$ must be proportional
to $\omega_2\wedge\omega_2$ and $\omega_6$ must be proportional
to $\omega_2\wedge\omega_2\wedge\omega_2$. This already rules out all the points of 
$W_0$ except the multiples $x_0=e_{1237}+e_{4567}$. But recall that a spinor $x$ 
is pure when the space of elements $v\in V_{14}$ such that $vx=0$ is seven dimensional.
A straightforward check shows that $x_0$ is only killed by (multiples of) $e_7$, 
hence is not pure. 
\qed 

\medskip 
An immediate consequence is:

\begin{coro} 
The Chow ring of $DG$ is free of rank $36$. 
\end{coro}

\smallskip
Explicitly, the $T$-fixed points correspond to the weight vectors in $\Delta$ 
of type $e_{ij}$, $e_{ii'j7}$, $e_{ijj'7}$, $e_{ii'jj'}$ where $1 \le i,i'\le 3$ 
and $4 \le j,j'\le 6$. Note that two fixed points $e_{ij}$ (respectively $e_{ijkl}$)
and $e_{abcd}$ are connected 
by a $T$-stable line if and only if $\{i,j\}\subset \{a,b,c,d\}$
(respectively $\{a,b,c,d\}$ and $\{i,j,k,l\}$ have three elements in common). 


\subsection{Schubert varieties}

Since the maximal torus $T$ of $G_2\times G_2$ acts on $DG$ with finitely many fixed points, 
the Bialynicki-Birula decomposition yields, for any choice of a general rank one 
subtorus, a stratification of $DG$ into affine spaces, which is uniquely defined 
up to conjugation. The closures of those affine spaces will be called Schubert 
varieties. Their classes in the (equivariant) Chow ring, called the (equivariant) 
Schubert classes, form a basis. A priori, we should be able to describe these
equivariant Schubert classes by localization, 
and then their multiplication rule. A more modest
goal would be to compute a Pieri formula in the classical Chow ring. This would
allow to get the degrees of the Schubert varieties, which would give lots of informations
on the restriction map from the spinor variety. In the case of $CG$, the restriction 
map from the ambient Grassmannian is surjective, so the multiplicative structure 
of the Chow ring of $CG$ can be deduced. 

\smallskip
In the case of a wonderful compactification $\bar{G}$ of an adjoint semisimple group $G$, 
the Schubert classes are indexed
by $W\times W$ and the Betti numbers are given by the following formula:
$$b_{2i}(\bar{G})=\#\{(u,v)\in W\times W, \; \ell(u)+\ell(v)+m(v)=i\},$$
where $\ell$ is the classical length function, and $m$ is the simple length function, 
defined as the number of simple roots that are sent to negative roots \cite{brioninfinity}. 
Recall that 
the Weyl group of $G_2$ is isomorphic with the dihedral group $D_6$, and in 
particular has $12$ elements: two elements in each length from $1$ to $5$, 
and one element of length $0$ and $6$. All have simple length $1$, except the 
maximal one (whose simple length are $0$ and $2$).
This yields the even Betti numbers of $\bar{G}_2$:
$$b_{2\bullet}(\bar{G}_2)= 1, 2, 4, 8, 12, 16, 19, 20, 19, 16, 12, 8, 4, 2, 1.$$
In order to deduce the Betti numbers of $DG$, we just need to recall that $\bar{G}_2$
can be obtained by blowing-up $\QQ_5\times\QQ_5$ in $DG$. This modifies the Betti 
numbers by the Betti numbers of a $(\PP^2-\PP^0)$-bundle over $\QQ_5\times\QQ_5$. 
We readily deduce:

\begin{prop} The Poincar\'e polynomial of the variety $DG$ is
$$P_{DG}(t)=\frac{1-t^{12}}{1-t^2}(1+t^6+t^8+t^{10}+t^{12}+t^{18}).$$
\end{prop}

In other words the odd Betti nubers of $DG$ are zero, and the even ones are 
$$b_{2\bullet}(DG)= 1, 1, 1, 2, 3, 4, 4, 4, 4, 4, 3, 2, 1, 1, 1.$$

Note that, as a consequence, the restriction map from $\SS_{14}$ cannot be surjective in degree four. 
In fact there is an obvious special  cohomology class of degree four, that of the closed orbit 
$\QQ_5\times \QQ_5$. Its degree is $4\binom{10}{5}=1008$, while the degrees of the 
restrictions to $DG$ of the degree four Schubert classes can be computed to be 
$$\int_{DG}\tau_4h^{10}=1260, \qquad \int_{DG}\tau_{31}h^{10}=1780.$$
So the class of $\QQ_5\times \QQ_5$ is certainly not an integral combination of 
the restrictions of $\tau_4$ and $\tau_{31}$, and probably not a combination at all. 

\medskip\noindent {\it Question}.
By pull-back, the Chow ring of $DG$ embeds inside the Chow ring of $\overline{DG}$.
Moreover, $\overline{DG}$ being the wonderful compactification of $G_2$, its equivariant 
cohomology ring can be extracted from \cite{strickland} or \cite{bj}. Can we deduce that of $DG$?
The Bialynicki-Birula decomposition of the wonderful compactification has been studied in 
\cite{brioninfinity}. Can one extract a Pieri formula, and push it down to $DG$? 

\section{Some incidences}

\subsection{Incidences for the Cayley Grassmannian}
Let us briefly consider the Cayley Grassmannian $CG\subset G(4,V_7)$, defined by the 
general three-form $\omega$.  The latter also defines a global section of $Q^\vee(1)$ 
on $G(2,V_7)$, whose zero locus is the adjoint variety of $G_2$. Consider the incidence
diagram
\begin{equation*}\label{I10}
\xymatrix{& CI_{10}\ar[dl]_p\ar[dr]^q & \\
CG  & & G(5,V_7)
}
\end{equation*}
where $CI_{10}$ parametrizes the pairs $(U_4\subset U_5)$ such that $U_4$ belongs to $CG$. 
In particular $CI_{10}$ is a $\PP^2$-bundle over $CG$. For $U_5\subset V_7$, the restriction 
of $\omega$ to $U_5$ is dual to a skew-symmetric degree two tensor which can be of rank two or four. In the latter case, the support of this tensor is a hyperplane $U_4\subset U_5$
on which $\omega$ vanishes, and it is the only such hyperplane; this implies that $q$ is
birational. The former case occurs over a locus $X_7$ of codimension three, and the 
corresponding fibers of $q$ are projective planes. We conclude that $q$ is just 
the blowup of $X_7\simeq OG(2,V_7)$. 

\smallskip
There is a slightly different incident diagram
\begin{equation*}\label{I11}
\xymatrix{& CI_{11}\ar[dl]_r\ar[dr]^s & \\
CG  & & X_{ad}(G_2)\subset G(2,V_7)
}
\end{equation*}
where the fibers of $s$ are del Pezzo fourfolds of degree five, and the fibers of $r$
are conics in $X_{ad}(G_2)$. As observed by Kuznetsov, this allows to interprete the Cayley 
Grassmannian $CG$ as the Hilbert scheme of conics on the adjoint variety of $G_2$.

\subsection{Incidences with $DG$} What are the analogs of those incidences when we switch to $DG$? Recall that 
$DG$ is defined by a general element of $\Delta$, which defines a global section of the
irreducible homogeneous vector bundle $\cE_{\omega_6}=U\otimes\cL$ over $\SS_{14}$. 
Over each flag variety $F$ of $Spin_{14}$, there is an irreducible homogeneous vector 
bundle $\cE_{\omega_6}^F$ whose space of sections is $\Delta$. 

Consider for example the flag variety $OF=OF(k,7,V_{14})$ for $k\le 5$, with its two projections 
to $\SS_{14}$ and $OG=OG(k,V_{14})$. 
The ranks of $\cE_{\omega_6}^{OF}$ and $\cE_{\omega_6}^{OG}$ can be read on the following weighted Dynkin diagram (where $k=3$):
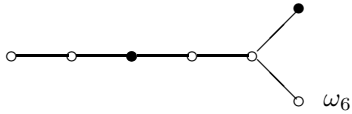
\begin{center}
\setlength{\unitlength}{4mm}

\begin{picture}(20,4)(2,-1.5)
\multiput(5,0)(2,0){5}{$\circ$}
\multiput(5.4,.25)(2,0){4}{$\line(1,0){1.7}$}
\put(9,0){$\bullet$}
\put(13.4,.4){\line(1,1){1.3}} 
\put(13.4,.1){\line(1,-1){1.3}} 
\put(14.55,-1.5){$\circ$}
\put(14.55,1.6){$\bullet$}
\put(15.6,-1.45){$\omega_6$}
\end{picture} 
\end{center}

The flag variety $OF$ is defined by the two marked vertices. When we suppress those two
vertices, the connected component of the remaining diagram containing the vertex associated
to $\omega_6$ has type $A_{6-k}$. So $\cE_{\omega_6}^{OF}$, which corresponds to the natural
representation, has rank $7-k$. Similarly, the orthogonal Grassmannian $OG$ is defined by the 
rightmost of the two marked vertices. When we suppress this vertex, the connected component 
of the remaining diagram containing the vertex associated to $\omega_6$ has type $D_{7-k}$. So $\cE_{\omega_6}^{OG}$, which corresponds to a half-spin representation, has rank $2^{6-k}$.

Our general element $z\in\Delta$ defines a general section $s_z$ of the bundle  $\cE_{\omega_6}^{OF}$, 
whose zero locus we denote by $OF_z$. The fibers of the projection to $\SS_{14}$ are Grassmannians
$G(k,7)$, and the restriction of $\cE_{\omega_6}^{OF}$ to each fiber is isomorphic with the quotient 
tautological bundle. In particular, if the restriction of $s_z$ to such a fiber is non identically zero, it vanishes on a copy of $G(k-1,6)$. So the general fiber of the projection from $OF_z$
to $\SS_{14}$ is $G(k-1,6)$, and the special fiber is $G(k,7)$ over $DG$. 

Similarly the projection of $OF$ to $OG$ is a spin manifold $\SS_{14-2k}$, and the restriction of $\cE_{\omega_6}^{OF}$ to each fiber is isomorphic to a spinor bundle. The zero-locus of $s_z$ to
such a fiber depends on its type as an element of the half-spin representation of $Spin_{14-2k}$. 
In fact this representation has finitely many orbits, so there is an induced stratification 
of $OG$ by orbital degeneracy loci of $s_z$, and the type of the fiber of the projection 
from $OF_z$ to $OG$ depends on the strata. 
Let us discuss two cases a little further. 

\subsection{Incidence with $4$-planes}
The case where $k=4$ is special because $Spin_{6}=SL_4$, and in this case the bundle 
$\cE_{\omega_6}^{OG}$ is just a rank four bundle defined by a natural representation 
of $SL_4$, as can be read from the weighted diagram 
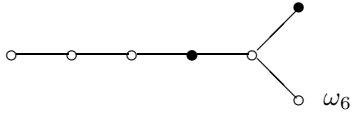
\begin{center}
\setlength{\unitlength}{4mm}

\begin{picture}(20,4)(2,-1.5)
\multiput(5,0)(2,0){5}{$\circ$}
\multiput(5.4,.25)(2,0){4}{$\line(1,0){1.7}$}
\put(11,0){$\bullet$}
\put(13.4,.4){\line(1,1){1.3}} 
\put(13.4,.1){\line(1,-1){1.3}} 
\put(14.55,-1.5){$\circ$}
\put(14.55,1.6){$\bullet$}
\put(15.6,-1.45){$\omega_6$}
\end{picture} 
\end{center} 

Similarly $\cE_{\omega_6}^{OF}$ is defined by a natural representation 
of $SL_3$, so on each fiber of the projection from $OF_z$ to $OG$, the section $s_z$ 
vanishes either at one point, or everywhere. We thus get a diagram 
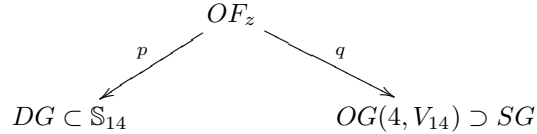
\begin{equation*}\label{I30}
\xymatrix{& OF_z\ar[dl]_p\ar[dr]^q & \\
DG\subset \SS_{14}  & & OG(4,V_{14})\supset SG 
}
\end{equation*}

\noindent
where $q$ is the blowup of a codimension four subvariety $SG\subset OG(4,V_{14})$, 
while $p$ is a $G(3,6)$-fibration over the complement of $DG$ in $\SS_{14}$, with special fibers $G(4,7)$ over $DG$. The weights of the rank four bundle $\cE_{\omega_6}^{OG}$ are $\omega_6$, $s_6(\omega_6)=\omega_5-\omega_6$, $s_5s_6(\omega_6)=\omega_4-\omega_5+\omega_7$ and 
$s_7s_5s_6(\omega_6)=\omega_4-\omega_7$, hence $\det (\cE_{\omega_6}^{OG})=\mathcal{O}(2)$. 
We readily deduce:

\begin{prop} The variety $SG$ is a Fano manifold of dimension $26$, Picard number $1$, and index $7$, 
admitting an action of $G_2\times G_2$. Its Poincar\'e polynomial is 
$$P_{SG}(t)=\frac{1-t^{10}}{1-t^2}(1+t^6)^2\Big (\frac{1-t^{16}}{1-t^4}(1+t^8+t^{10}+t^{12}+t^{20})+t^{16}\Big).$$
\end{prop}

This means the odd Betti numbers of $SG$ are zero, and the even ones are 
$$b_{2\bullet}(SG)= \scriptstyle{1, 1, 2, 4, 6, 8, 12, 16, 20, 25, 29, 33, 35, 36, 35, 33, 29, 25, 20, 16, 12, 8, 6, 4, 2, 1, 1.}$$
The topological Euler characteristic is $420$. It would be interesting to know if the action 
of $G_2\times G_2$ is quasi-homogeneous. 

\subsection{Incidence with $2$-planes}
Over the orthogonal Grassmannian $OG(2,14)$, 
the bundle $\cE^{OG}_{\omega_6}$  has rank $16$ and is induced from a half-spin representation of $Spin_{10}$. 
Since $OG(2,14)$ has dimension $21$, the general section of $\cE^{OG}_{\omega_6}$ defined by $z$ must vanish in dimension $5$ (or possibly, nowhere), and its zero locus $Z_z$ must be 
stable under the action of 
$G_2\times G_2$. 

\begin{prop} $Z_z$ is the disjoint union of two copies of $X_{ad}(G_2)$. \end{prop}

\begin{proof} 
Recall that our general element $z$ of $\Delta$ determined an orthogonal decomposition 
$V_{14}=V_7\oplus V'_7$ and a tensor decomposition $\Delta=\Delta_7\otimes\Delta'_7$ 
such that $z=\delta\otimes\delta'$ for some general elements $\delta$ and $\delta'$
of $\Delta_7$ and $\Delta'_7$.  

Given an orthogonal plane $P$, consider the Plücker line $\wedge^2P$. The image  
of the Clifford multiplication map 
$$\wedge^2P\otimes\Delta \subset \wedge^2V_{14}\otimes\Delta \lra\Delta$$
is a sixteen dimensional space $\cG_P\subset\Delta$, and we can identify $\cG$ with 
$\cF^\vee$ (recall that $\Delta$ is self-dual). This implies that $P$ belongs to 
$Z_z$ if and only if $\cG_P\subset\omega^\perp$. 

Now suppose that $P\subset V_7$. The Clifford action of $P$ on  $\Delta=\Delta_7\otimes\Delta'_7$
is just given by its action of $\Delta_7$, so we deduce that $P$ belongs to $Z_z$ if and
only if $\wedge^2P.\Delta_7\subset\delta^\perp$. This is a codimension two condition on
$OG(2,V_7)$, that defines the adjoint variety $X_{ad}(G_2)$.

We conclude that  $Z_z$ contains the disjoint union of $X_{ad}(G_2)$ and $X'_{ad}(G_2)$, the adjoint varieties of our two copies of $G_2$.
In order to prove equality, we just need to check that  $Z_z$ has at most two connected
components. For this we can use the Koszul resolution of the structure sheaf of $Z_z$. 
A direct computation shows that the only non zero cohomology groups of the wedge powers 
of the dual of 
$\cE_{\omega_6}^{OG}$ are $H^0(\wedge^0(\cE_{\omega_6}^{OG})^\vee)=H^4(\wedge^4(\cE_{\omega_6}^{OG})^\vee)=\CC$. We readily deduce that $h^0(\cO_{Z_z})=2$, and this concludes the proof. 
\end{proof}

Taking the incidence between $DG$ and $Z_z$ we get the following diagram:
\begin{equation*}\label{I13}
\xymatrix{& & DI_{13}\ar[dl]_{2:1}\ar[dr]^t & \\
DG& D\ar@{^{(}->}[l] & & X_{ad}(G_2)\amalg X'_{ad}(G_2)
}
\end{equation*}
where the fibers of $t$ are codimension two linear sections of $\SS_{10}$.

\section{Linear subspaces}

Since $DG$ has dimension $14$ and index $7$, the expected dimension of the space 
of lines on $DG$ is $14+7-3=18$. The expected dimension of the space of lines
through a general point, or of the VMRT, is $5$. 

\begin{prop}
The variety $F_1(DG)$ of lines on $DG$ is a smooth Fano manifold of dimension $18$, 
Picard number one, and index $4$.
\end{prop}

\proof The variety of lines on $\SS_{14}$ is the orthogonal Grassmannian 
$OG(5,14)$, whose dimension is $30$. The weights $\omega_1,\omega_6,\omega_7$
define irreducible homogeneous vector bundles of ranks $5$, $2$, $2$ on 
$OG(5,14)$: the first one is $\cV^\vee$, the dual of the tautological 
bundle, and we denote the other ones by $\cE_6$ and $\cE_7$. Their determinant 
line bundles are all equal to $\cO(1)$, the restriction of the Pl\"ucker line bundle. 
Note moreover that 
$$\cE_6\otimes\cE_7\simeq (\cV^\perp/\cV)(1).$$
Consider the incidence diagram, where $OF(5,7,14)=D_7/P_{5,7}$, 
\begin{equation*}\label{linesDG}
\xymatrix{& \cU\otimes\cL\ar[d] & OF(5,7,14)\ar[ld]_p\ar[dr]^q & \cF\ar[d] \\
DG\ar@{^{(}->}[r] &\SS_{14}  & & OG(5,14) &  F_1(DG)\ar@{^{(}->}[l] 
}
\end{equation*}
Since $DG$ is defined by a general section $s$ of the bundle $\cE=\cU\otimes\cL$ on 
$\SS_{14}$, its variety of lines $F_1(DG)$ will be defined by a section of
the bundle $\cF=q_*p^*\cE$ on $OG(5,14)$. Obviously there is an exact sequence
$$0\lra  q^*\cV\otimes p^*\cL\lra p^*(\cU\otimes \cL)\lra (p^*\cU/q^*\cV)\otimes p^*\cL\lra 0$$
on $OF(5,7,14)$. We claim that this pushes forward on $OG(5,14)$ to 
$$0\lra  \cV\otimes \cE_6\lra \cF\lra \cE_7\lra 0.$$
We deduce that $\cF$ has rank $12$, and that the space of its global sections is again $\Delta$. By construction $\cF$ is globally generated, so $F_1(DG)$ is smooth of dimension
$30-12=18$, being the zero-locus of a general section. Since moreover $\det \cF=\cO(4)$, 
we deduce that $F_1(DG)$ is Fano of index $4$.

In order to check that  $F_1(DG)$ has Picard number one, consider the point-line
incidence correspondence 
\begin{equation*}\label{I19}
\xymatrix{& I_{19}\ar[dl]_p\ar[dr]^q & \\
DG  & & F_1(DG). 
}
\end{equation*}

\noindent 
Of course $q$ is just a $\PP^1$-bundle. The fibers of $p$ are of three different types, 
over the three orbits in $DG$. A computation shows that the fiber over the closed orbit 
is the union of two copies of $\PP^2\times \PP^3$ blown-up at one point, while the 
other orbits are irreducible. We could in principle compute the Hodge polynomials of
the three fibers and deduce that of $F_1(DG)$, but the simple fact that the fiber
over the codimension one orbit $\cO_1\subset DG$ is irreducible already implies that 
the Picard number of $F_1(DG)$ is one, as claimed. 
\qed

\medskip
The generic fiber of $p$ is the variety of lines in $DG$ through a general point. 
It is isomorphic with its image in the tangent space, the {\it variety of minimal
rational tangents} (VMRT).

\begin{prop}
The VMRT at a general point of $DG$ is a copy of the adjoint variety $X_{ad}(G_2)\subset\PP\fg_2$.
\end{prop}

\proof The general point $x$ of $DG$ has stabilizer $G_2$, 
so the VMRT at $x$ is a five dimensional subvariety, stable under $G_2$, 
and equivariantly embedded inside $\PP T_xDG=\PP^{13}$. This VMRT must contain
a closed $G_2$-orbit, and it contains no fixed point because the restriction 
of $D_z=V_7\oplus V'_7\oplus\CC$ to the diagonal $G_2$ contains a unique stable plane,
but the corresponding line, since it contains $[z]$, is not contained in $DG$. Since 
the minimal non trivial closed $G_2$-orbits are $G_2/P_1=\QQ_5$ and $G_2/P_2=X_{ad}(G_2)$;
both of dimension five, the VMRT must be one of these. Since it is equivariantly
embedded inside $\PP^{13}$, it must be the second one. \qed

\smallskip It was already observed in \cite{bf} that the VMRT at a general point of 
the wonderful compactification of an adjoint simple algebraic group  is a copy
of its adjoint variety (except in type $A$). The only special feature in our 
situation is that the minimal rational curves are lines in the spinor variety. 

\begin{coro}
$DG$ contains planes, but no higher dimensional linear spaces, 
passing through the general point. 
\end{coro}

In fact we have seen that $DG$ also contains a ten-dimensional family of $\PP^3$'s, 
parametrized by $X_{ad}(G_2)\times X_{ad}(G_2)$. But they only cover the 
codimension one orbit closure (and there is exactly one of them through
the general point). 
 
\section{Some numerology}
Let us conclude this paper by a couple of slightly esoteric observations.
The Cayley Grassmannian and its double appear in two series of compactifications of 
symmetric spaces, as follows: 
$$\begin{array}{cccc}
X= & SL_3/SO_3\subset\PP^5, &  SO_5/GL_2\subset\QQ^3\times\QQ^3, & G_2/SO_4\subset CG, \\ 
Y= & PSL_3\subset\PP^8, & SO_5\subset\SS_{10}, & G_2\subset DG.
\end{array}$$
Each of these compactifications contains a unique closed orbit, and blowing it yields
the wonderful compactification. Let $a=1,2,4$ for the three members of each series. 
The closed orbit $Z$ in the first series has dimension $a+1$ and codimension $3$. 
The closed orbit $Z'$ in the second series is isomorphic with 
$Z\times Z$ , so its dimension is $2a+2$, while its codimension is  $4$. In fact each $Z$
in the series admits 
a homogeneous rank two vector bundle $C$ such that its normal bundle is isomorphic with
$Sym^2C$, while the normal bundle to $Z'$ is isomorphic with $C\boxtimes C$. 

The Weyl group $W$ has cardinality $2a+4$. Recall that the Chow ring of the wonderful
compactification has a basis indexed by $W\times W$, so that the Euler topological 
characteristic $\chi_{top}(\bar{G})=(\# W)^2 =4(a+2)^2$. A computation shows that 
the minimal compactification $Y$ as Euler characteristic 
$$\chi_{top}(Y)=\frac{1}{4}\chi_{top}(\bar{G})=(a+2)^2. $$
Does it admit a natural basis indexed by $\bar{W}\times \bar{W}$, where $\bar{W}=W/\ZZ_2$?


\smallskip 

\begin{thebibliography}{Aa}

\bibitem{am} Abuaf R., Manivel L., {\it Gradings of Lie algebras, magical spin geometries 
and matrix factorizations}, arXiv:1901.07252.

\bibitem{allison}
Allison B.N., {\it
A class of nonassociative algebras with involution containing the class of Jordan algebras},
Math. Ann. {\bf 237} (1978), no. 2, 133–156.

\bibitem{bh} Baez J., Huerta J., {\it $G_2$ and the rolling ball},
Trans. Amer. Math. Soc. {\bf 366} (2014), 5257--5293. 

\bibitem{fuman}  Bai C., Fu B., Manivel L., {\it 
On Fano complete intersections in rational homogeneous varieties},
arXiv:1808.01549, to appear in Math. Zeitschrift.

\bibitem{bm} Benedetti V., Manivel L., {\it 
The small quantum cohomology of the Cayley Grassmannian},
arXiv:1907.07511, to appear in Int. J. Math.

\bibitem{hillerboe}
Boe B., Hiller H., {\it 
Pieri formula for $SO_{2n+1}/U_n$ and $Sp_n/U_n$,} 
Advances in Math. {\bf 62} (1986), no. 1, 49--67. 

\bibitem{brioninfinity} Brion M., {\it 
The behaviour at infinity of the Bruhat decomposition},
Comment. Math. Helv. {\bf 73} (1998), no. 1, 137–174. 

\bibitem{bf} Brion M., Fu B., {\it 
Minimal rational curves on wonderful group compactifications},
J. Éc. polytech. Math. {\bf 2} (2015), 153–170. 


\bibitem{bj} Brion M., Joshua R., {\it
Equivariant Chow ring and Chern classes of wonderful symmetric varieties of minimal rank},
Transform. Groups 13 (2008), no. 3-4, 471–493.

\bibitem{chevalley}
Chevalley C., The algebraic theory of spinors and Clifford algebras, in Collected works, 
Vol. 2, Springer 1997.

\bibitem{hm}
Hwang J.-M., Mok N. , {\it Prolongations of infinitesimal linear automorphisms of projective 
varieties and rigidity of rational homogeneous spaces of Picard number 1 under Kähler deformation},
Invent. Math. {\bf 160} (2005), no. 3, 591–645.

\bibitem{pk}
Kim S.-Y, Park K.-D.,  {\it On the deformation rigidity of smooth projective symmetric 
varieties with Picard number one}, C. R. Math. Acad. Sci. Paris {\bf 357} (2019), no. 11-12, 889–896. 

\bibitem{kuzspin10}
Kuznetsov A. G., {\it On linear sections of the spinor tenfold} I, 
Izv. Math. {\bf 82} (2018), no. 4, 694–751.

\bibitem{sext}
Landsberg J.M., Manivel L., {\it The sextonions and} $E_{7\frac{1}{2}}$,  Advances in
Math. {\bf 201} (2006), no. 1, 143–179. 

\bibitem{cayley}
Manivel L., {\it  The Cayley Grassmannian},  J. Algebra {\bf 503} (2018), 277--298. 

\bibitem{ott}
Ottaviani G., {\it 
On Cayley bundles on the five-dimensional quadric},
Boll. Un. Mat. Ital. A {\bf 4} (1990), no. 1, 87-100.

\bibitem{ruzzi}
Ruzzi A. {\it  Geometrical description of smooth projective symmetric varieties with Picard number one},  Transform. Groups {\bf 15} (2010), no. 1, 201--226.

\bibitem{sk}
Sato M., Kimura T., {\it A classification of irreducible prehomogeneous vector spaces 
and their relative invariants}, Nagoya Math. J. {\bf 65} (1977), 1–155.

\bibitem{strickland}
Strickland E., {\it  
Equivariant cohomology of the wonderful group compactification}, 
J. Algebra 306 (2006), no. 2, 610--621. 
\end{thebibliography}
\end{document}